\documentclass[preprint,12pt]{article}
\usepackage{amsthm,amsmath,amssymb,latexsym}
\usepackage{mathtools}
\usepackage{enumerate}
\usepackage{framed}
\usepackage{tikz}
\usepackage{tikz-cd}
\usepackage{tabularx}
\usepackage{array}
\usepackage{todonotes}
\usepackage{pdflscape}
\usepackage{setspace}
\usepackage{tablefootnote}
\usepackage{anysize}
\usepackage{versions}
\usepackage{hyperref}
\usepackage{cancel}
\usepackage{yhmath}
\usepackage{physics}
\usepackage{mathdots}
\usepackage{color}
\usepackage{siunitx}
\usepackage{multirow}
\usepackage{gensymb}
\usepackage{booktabs}
\usepackage{float}
\usepackage{lineno}

\marginsize{1in}{1in}{1in}{1in}

\newcolumntype{M}[1]{>{\centering\arraybackslash}m{#1}}

\usetikzlibrary{babel}

\renewcommand{\le}{\leqslant}
\renewcommand{\ge}{\geqslant}

\usepackage{versions}
\usepackage{hyperref}
\hypersetup{
  colorlinks   = true, 
  urlcolor     = blue, 
  linkcolor    = blue, 
  citecolor   = red 
}

\usetikzlibrary{fadings}
\usetikzlibrary{patterns}
\usetikzlibrary{shadows.blur}
\usetikzlibrary{shapes}
\usetikzlibrary{positioning}
\usetikzlibrary{3d, shapes.geometric, calc}

\newtheorem{theorem}{Theorem}[section]
\newtheorem{proposition}[theorem]{Proposition}
\newtheorem{lemma}[theorem]{Lemma}

\newtheorem{corollary}[theorem]{Corollary}
\newtheorem{question}[theorem]{Question}
\newtheorem*{question*}{Question}
\newtheorem{example}[theorem]{Example}

\theoremstyle{definition}
\newtheorem{definition}[theorem]{Definition}
\theoremstyle{remark}
\newtheorem{remark}[theorem]{Remark}

\numberwithin{equation}{section}

\DeclareMathOperator{\ind}{ind}

\DeclareMathOperator{\xind}{x-ind}

\DeclareMathOperator{\sind}{sim-ind}
\DeclareMathOperator{\sd}{sd}

\renewcommand{\epsilon}{\varepsilon}
\renewcommand{\phi}{\varphi}
\renewcommand{\kappa}{\varkappa}
\renewcommand{\theta}{\vartheta}

\begin{document}

\title{The mapping index through the lens of the cross-index}

\author{Vuong Bui\thanks{
Swinburne Vietnam, FPT University, Hanoi, 80 Duy Tan Street, Hanoi 100000, Vietnam (\texttt{bui.vuong@yandex.ru})
},\; Hamid Reza Daneshpajouh\thanks{
School of Mathematical Sciences, University of Nottingham Ningbo China, 199 Taikang East Road, Ningbo, 315100, China (\texttt{Hamid-Reza.Daneshpajouh@nottingham.edu.cn})
},\; Roman Karasev\thanks{
Institute for Information Transmission Problems RAS, Bolshoy Karetny per. 19, Moscow, Russia 127994 and Moscow Institute of Physics and Technology, Institutskiy per. 9, Dolgoprudny, Russia 141700 (\texttt{r\_n\_karasev@mail.ru}) \\
The research of R. Karasev was carried out within the state assignment 1.1.1-0029/25 of Ministry of Science and Higher Education of the Russian Federation for IITP RAS
}
}
\date{}
\maketitle

\begin{abstract}
We study the cross-index of free \(G\)-posets as a combinatorial analogue of the equivariant topological index. We demonstrate that the cross-index exhibits many structural properties closely paralleling those of the topological index, while its behavior with respect to unions displays a pronounced dichotomy depending on the acting group. Specifically, if \(P = A \cup B\) is a union of \(G\)-invariant subposets, then for \(G = \mathbb{Z}_2\) we obtain the sharp inequality \(
\xind P \le \xind A + \xind B + 1, \)
which is directly analogous to the classical union inequality for the topological index. In contrast, for every group \(G\neq \mathbb{Z}_2\), this phenomenon fails in general, and we establish the best possible weaker estimate
\(
\xind P \le \xind A + 2(\xind B+1).
\)
This reveals a fundamental distinction between the \(\mathbb{Z}_2\)-equivariant and non-\(\mathbb{Z}_2\)-equivariant settings at the purely combinatorial level. As further consequences, we compare the cross-index with both the topological index and the simplicial index, showing in particular that the gap between the cross-index and the topological index can be arbitrarily large. These results clarify the role of the cross‑index as a combinatorial analogue of the equivariant topological index and further strengthen the interplay between equivariant topological methods and combinatorial structures endowed with symmetry.
\end{abstract}

\section{Introduction}

Equivariant topology is a fascinating and dynamic branch of mathematics that focuses on the study of topological spaces which possess certain symmetries. Its popularity arises from its broad applications across diverse mathematical disciplines, such as discrete geometry~\cite{vzivaljevic2017topological,blagojevic2017beyond}, graph~\cite{lovasz1978kneser,daneshpajouh2018new, chen2015multichromatic, daneshpajouh2021neighborhood}, and hypergraph coloring problems~\cite{alon1986chromatic, alishahi2015chromatic,daneshpajouh2023hedetniemi, frick2020chromatic}, as well as fair division problems~\cite{de2013course,avvakumov2021envy}. For a gentle introduction to equivariant methods in combinatorics, we refer the reader to the excellent book by Ji\v{r}\'{i} Matou\v{s}ek~\cite{matouvsek2003using}. Typically, these applications involve reducing original problems to the nonexistence of an equivariant map between two specific equivariant topological spaces. In order to be able to show that there is no equivariant map between some equivariant topological spaces, some ``topological measures'' are defined. Most of such measures are functions from the set of a class of equivariant topological spaces to non-negative integers with the ``monotonicity property'': The existence of an equivariant map between two spaces guarantees that the invariant value of the domain never exceeds that of the codomain. This fundamental property yields an effective obstruction principle: whenever the invariant value of a space strictly exceeds that of another space, one can immediately conclude that no equivariant map exists from the former to the latter. Among these measures, the topological index, denoted by $\ind X$ for a $G$-space $X$, is one of the most powerful. It is defined as the minimum value of $n$ for which a $G$-equivariant map (often referred to simply as a $G$-map) exists from space $X$ to $E_nG$. Here, $E_nG$ represents the standard $(n+1)$-fold (topological) join $\underbrace{G*G*\dots *G}_{n+1}$ with the diagonal action, where $G$ is a topological group (throughout this paper, when \(G\) is a finite group, it is equipped with the discrete topology).  Note that when $G=\mathbb{Z}_2$, $E_n\mathbb{Z}_2$ is ($\mathbb Z_2$-homeomorphic to) the $n$-dimensional sphere $\mathbb{S}^n$ with the antipodal action. In this paper, we aim to study the topological index through the lens of a combinatorial index known as the cross-index, which was originally introduced in~\cite{Sim}. This combinatorial index has proven to be both powerful and versatile, with applications spanning various combinatorial structures, including graph coloring problems~\cite{Sim,daneshpajouh2021colorings, daneshpajouh2025box}, hypergraph coloring~\cite{Ali}, and topological problems such as the generalized Topological Hedetniemi's Conjecture~\cite{bui2023topological}. Furthermore, the cross-index is clearly computable, whereas computability of the index itself remains unknown. To introduce our key findings, it is prudent to first revisit some foundational concepts associated with the generalized cross-index for arbitrary groups, as introduced in~\cite{bui2023topological}. Throughout this paper, $G$ refers to a non-trivial group (i.e., $|G| \ge 2$), and its identity element is denoted by $e$. A $G$-poset is defined as a partially ordered set $(P, \preceq)$ that is equipped with a group action by $G$ that preserves the order structure; specifically, if $p_1 \preceq p_2$, then $g p_1 \preceq g p_2$ for every $g \in G$. The group action is said to be free if the identity element is the only group member that fixes any point in $P$. A $G$-equivariant order-preserving map (or simply a $G$-map) between two $G$-posets $P$ and $Q$ is a function $\psi : P \to Q$ that is simultaneously order-preserving ($p_1 \preceq p_2 \Rightarrow \psi(p_1) \preceq \psi(p_2)$) and preserves the $G$-action ($\psi(g p) = g \psi(p)$ for all $g \in G$). The face poset $\mathcal{F}(\mathcal{K})$ of the simplicial complex $\mathcal{K}$ is the poset whose vertices are all non-empty simplices of $\mathcal{K}$ ordered with the inclusion. If $\mathcal{K}$ is $G$-simplicial complex, then we consider $\mathcal{F}(\mathcal{K})$  with the action naturally induced from $\mathcal{K}$. Finally, for any non-negative integer $n$, we define $Q_n G$ as the $G$-poset with underlying set $G \times\{0, \dots, n\}$, where the $G$-action is given by left multiplication on the first component ($g \cdot (h, i) = (g h, i)$) and the order relation satisfies $(g, i) \prec (h, j)$ precisely when $i < j$. Here, for each pair $(g,i) \in Q_n G$, we refer to $g$ as the \emph{sign} and $i$ as the \emph{value}. Now we are ready to recall the definition of cross-index for an arbitrary $G$-poset.

\begin{definition}[\cite{bui2023topological}]
The cross-index of a $G$-poset $P$, denoted $\xind P$, is the minimal integer $n$ for which a $G$-map $\psi : P \to Q_n G$ exists.
\end{definition}

We now return to the main thread of our discussion. It is well established that the topological index satisfies several fundamental structural properties. Motivated by these, we demonstrate that analogous results also hold for the cross-index. Furthermore, we show that the topological properties in question can be derived purely from combinatorial conditions, and we provide a simpler, entirely combinatorial proofs of these facts. There is, however, one notable exception to this correspondence, namely the ``union property''. Indeed, it is a classical fact that if \(X\) is a free \(G\)-simplicial complex and \(\mathcal{K}\) and \(\mathcal{L}\) are \(G\)-equivariant subcomplexes such that \(X = \mathcal{K} \cup \mathcal{L}\), then the following inequality holds:
\[
\ind X \le \ind \mathcal{K} + \ind \mathcal{L} + 1.
\]
This naturally leads to the question whether an analogous inequality holds for the cross-index.

\begin{question}
\label{que: que1}
Does the inequality \( \xind P \le \xind A + \xind B + 1 \) hold whenever \(P = A \cup B\), where \(A\) and \(B\) are \(G\)-invariant subposets of \(P\)?
\end{question}

We first observe that, for the important class of \(G\)-posets arising as face posets of \(G\)-simplicial complexes, the classical union inequality indeed remains valid. This shows that Question~\ref{que: que1} is natural and that the classical
bound remains valid for an important class of \(G\)-posets arising from
simplicial complexes.

\begin{theorem}\label{Thm: union1}
Let \(X = \mathcal{K} \cup \mathcal{L}\) be a free \(G\)-simplicial complex, where \(\mathcal{K}\) and \(\mathcal{L}\) are free \(G\)-invariant subcomplexes. Then
\[
\xind \mathcal{F}(X) \le \xind \mathcal{F}(\mathcal{K}) + \xind \mathcal{F}(\mathcal{L}) + 1.
\]
\end{theorem}

However, when we pass to the general setting of arbitrary \(G\)-posets, we encounter a more subtle phenomenon: the cross-index of a union depends in a crucial way on the cardinality of \(G\). We show that the behavior of the cross-index with respect to unions differs markedly for the case \(G = \mathbb{Z}_2\) in comparison with all other groups. More precisely, our main result yields a negative answer to Question~\ref{que: que1} whenever \(G\) is not \(\mathbb{Z}_2\), the cyclic group of order two, and a positive answer in the case \(G = \mathbb{Z}_2\).

\begin{theorem}\label{Main:Theorem 1}
Let \(P\) be a free \(G\)-poset, and let \(A\) and \(B\) be two \(G\)-invariant subposets such that their union covers \(P\), i.e., \(P = A \cup B\).

\begin{itemize}
\item[(a)] If \(G\) contains at least three elements, then
\[
\xind P \le \xind A + 2(\xind B + 1),
\]
\item[(b)] If \(G = \mathbb{Z}_2\) is the cyclic group of order two, then the bound improves to
\[
\xind P \le \xind A + \xind B + 1,
\]
which coincides with the corresponding topological estimate.
\end{itemize}
\end{theorem}

Furthermore, we will see that these bounds are sharp. Our next result is a partial answer to the problem, posed in \cite{bui2023topological}, of how large the gap between the cross-index and the topological index can be. 

\begin{question}[{\cite[Question 1]{bui2023topological}}]
\label{que: que2}
Given positive integers $m$ and $n$ with $m< n$, is there any finite free $G$-posets $P$ such that
$\ind\Delta P  = m$ but $\xind P= n$? 
\end{question}

Indeed, we show that:
\begin{theorem}\label{thm:large-gap}
For every integer $n\ge 1$ and a nontrivial group $G$, there exists a $G$-poset $Q$ such that $\xind Q = 3n-1$ and $\ind\Delta Q= 2n-1$.
\end{theorem}
Furthermore, as an application of Theorem~\ref{Main:Theorem 1}, we also observe that any \(G\)-poset \(P\) with
\[
\xind P > \xind A + \xind B + 1 \quad \text{for} \quad P = A \cup B,
\]
where \(A\) and \(B\) are \(G\)-invariant subposets of \(P\), yields another explicit construction of a poset for which the cross-index and the topological index differ; see Remark~\ref{remark: final} for details.


\section{Preliminary Observations}

In this section, we compare the properties of the topological index and simplicial index and investigate whether analogous properties hold for the cross-index. Through this analysis, we demonstrate that certain results for the topological index can be naturally recovered or extended using the cross-index framework. 

\subsection{Simplicial index compared to cross-index}

Let us compare the cross-index with the closely related notion of the simplicial index, which is defined in terms of face-wise linear maps into \(E_n G\) to obtain a better picture of cross-index.  As noted above, the topological index is the least $n$ for which a $G$-space admits a $G$-map into the $(n+1)$-fold join of $G$. We show below that the cross-index does not, in general, coincide with the least $n$ for which such a map can be chosen ``face-wise linear''. To make this precise, we first define the simplicial index.

\begin{definition} 
Let $\mathcal{K}$ be a free $G$-simplicial complex. The simplicial index $\sind \mathcal{K}$ is the least $n$ such that there exists a $G$-simplicial map $\mathcal{K} \to E_n G$.\footnote{In~\cite{avvakumov2021systolic} a special case of the simplicial index called \emph{combinatorial essentiality} there was used to infer lower bounds on the number of vertices of a simplicial complex.}
\end{definition}

By definition, the topological index is monotone with respect to $G$-maps: if there exists a $G$-equivariant map $X \to Y$, then $\ind X \le \ind Y$. The same monotonicity holds for the cross-index and the simplicial index, since compositions of $G$-maps remain $G$-maps in each of these categories. Recall that the order complex $\Delta P$ of a poset $P$ is the simplicial complex whose simplices are the nonempty chains in $P$. If $P$ is a $G$-poset, then $\Delta P$ inherits a natural structure of a $G$-simplicial complex via the induced action. Any $G$-map $f \colon P \to Q$ of $G$-posets induces a $G$-equivariant simplicial map $\Delta f \colon \Delta P \to \Delta Q$. Combining this with the definition of the simplicial index and the fact that $\Delta Q_n= E_n G$, we obtain, for every $G$-poset $P$, 

\begin{equation}\label{eq:index_inequality 1} 
\sind \Delta P \le \xind P.
\end{equation}
Since any $G$-simplicial map is, in particular, a continuous map, we also have 
\begin{equation}\label{eq:index_inequality 2}
\ind \Delta P \le \sind \Delta P,
\end{equation}
and hence 
\begin{equation}\label{eq:index_inequality 3} 
\ind \Delta P \le \xind P. 
\end{equation}

Theorem~\ref{thm:large-gap} shows that the gap between $\ind \Delta P$ and $\xind P$ can be arbitrarily large for suitable families of examples. On the other hand, equality can be forced after sufficiently many subdivisions of $\Delta P$:

\begin{proposition}[{\cite[Proposition 4]{bui2023topological}}]\footnote{An analogous statement for the simplicial index follows from the equivariant simplicial approximation theorem.}
\label{pro:Approximation} 
For any free $G$-poset $P$ there exists nonnegative integer $r$ such that 
$$\xind \mathcal{F}(\sd^r(\Delta P))=\ind \Delta P,$$
where $\sd^r(\Delta P)$ denotes the r-fold barycentric subdivision of $\Delta P$.
\end{proposition}

Moreover, for every $G$-simplicial complex $\mathcal{K}$ one has 

\begin{equation}\label{eq:index_inequality_4}
\xind \mathcal{F}(\mathcal{K}) \le \sind \mathcal{K}, 
\end{equation} 
since any $G$-simplicial map $\mathcal{K} \to \mathcal{L}$ naturally induces a $G$-order-preserving map $\mathcal{F}(\mathcal{K}) \to \mathcal{F}(\mathcal{L})$, and there is a natural $G$-order-preserving map $\mathcal{F}(E_n G) \to Q_n G$; see Proposition~\ref{prop:dimension_bound}. 

We now give a simple argument showing that the difference between the simplicial index of a simplicial complex and the cross-index of its face poset may be arbitrarily large.

\begin{proposition}
\label{proposition:chromatic}
For any group $G$, the gap between the simplicial index $\sind \mathcal{K}$ of a simplicial complex and the cross-index of its face poset $\xind \mathcal{F}(\mathcal{K})$ can be arbitrarily large for certain free $G$-simplicial complex $K$. 
\end{proposition}
\begin{proof}
Fix a group $G$ and a finite graph $K$. Define the graph $K \times G$ to have vertex set $V(K) \times G$, with an edge between $(k,h)$ and $(k',h')$ if and only if $k \ne k'$ and $kk'$ is an edge of $K$. View $K \times G$ as a 1-dimensional $G$-simplicial complex, where $G$ acts freely by left translation on the second coordinate. Consider a $G$-simplicial map 
$$\Phi \colon K \times G \longrightarrow E_n G$$ 
Such a map is determined by the image of a representative of each orbit, e.g., by the images of the vertices $(o,e)$ with $o \in V(K)$. Simpliciality forces adjacent vertices in $K \times G$ to map to distinct join factors of $E_n G$. Consequently, if $o,o'$ are adjacent in $K$, then $\Phi(o,e)$ and $\Phi(o',e)$ lie in different factors, and therefore $n+1$ must be at least the chromatic number $\chi(K)$. Conversely, given a proper coloring $\phi \colon V(K) \to \{0,1,\dots,m\}$ of $K$, we obtain a $G$-simplicial map 
$$K \times G \longrightarrow E_m G,\qquad (k,g) \longmapsto (g, \phi(k)),$$
where the first coordinate specifies the join factor. It follows that $\sind K \times G = \chi(K) - 1.$ Taking $K = K_{n+1}$ to be the complete graph on $n+1$ vertices yields $\sind K_{n+1} \times G = n$. On the other hand, the cross-index in question is always bounded from above by the height of the face poset (see Proposition~\ref{prop:dimension_bound}). For a graph, the face poset has height 1, hence $\xind \mathcal{F}(K \times G) \le 1$, and consequently $\ind (K \times G) \le 1$. 
\end{proof}

Question~\ref{que: que2} and inequality~\eqref{eq:index_inequality 1} naturally lead to another question of whether the gap between $\xind P$ and $\sind \Delta P$ can also be arbitrarily large. The main result of this paper, Theorem~\ref{Main:Theorem 1}, surprisingly shows that this is not the case. Before presenting the argument, we make simple observations about the equivalent formulations of the cross-index and the simplicial index in order to clarify their relationship. According to the definitions of cross-index and simplicial index, a $G$-poset (respectively, a $G$-simplicial complex) has cross-index (respectively, simplicial index) equal to zero if, in the comparability graph of it (respectively, in the $1$-skeleton of the complex), there is no path between any two distinct elements (vertices) belonging to the same $G$-orbit. Consequently:

\begin{proposition}
\label{proposition:partition-sind}
For a $G$-poset $P$, $\sind \Delta P= m$ means that $m$ is the smallest nonnegative integer for which there exists a $G$-invariant partition
\[
P \;=\; A_0 \sqcup A_1 \sqcup \dots \sqcup A_m
\]
such that each $A_i$ has index zero.
\end{proposition} 
\begin{proof}
Indeed, a $G$-simplicial map $\phi\colon \Delta P\to E_m G$ witnessing $\sind \Delta P=m$ induces such a natural partition by setting
\[
A_i \;=\; \phi^{-1}\bigl(\{(g,i): g\in G\}\bigr).
\]
Conversely, suppose
\[
P=A_0\sqcup A_1\sqcup\cdots\sqcup A_m
\]
is a \(G\)-invariant partition such that each \(A_i\) has index zero. Then for every \(i\) there exists a
\(G\)-simplicial map
\[
\psi_i\colon \Delta A_i\to E_0G=G .
\]
Then one can define a
\(G\)-simplicial map \(\phi\colon \Delta P\to E_mG\) by
\[
\phi(p):=(\psi_i(p),i)\qquad\text{for }p\in A_i.\qedhere
\]
\end{proof}

\begin{proposition}
\label{proposition:partition-xind}
For a $G$-poset $P$, $\xind P = m$ means that $m$ is the smallest nonnegative integer for which there exists a $G$-invariant partition $P \;=\; A_0 \sqcup A_1 \sqcup \dots \sqcup A_m$ with the additional property that whenever $a_i\in A_i$ and $a_j\in A_j$, the situation
\[
i<j \quad\text{and}\quad a_i \succ a_j
\]
is impossible in the poset order.
\end{proposition}
\begin{proof}
Similar to the proof of Proposition~\ref{proposition:partition-sind}.
\end{proof}

Now, by combining the principal result of this paper, Theorem~\ref{Main:Theorem 1}, with Inequality~\ref{eq:index_inequality 1}, we obtain the following corollary (see proof in~\ref{proof: cor main result}).

\begin{corollary}
\label{cor: main result}
If $P$ is a free $G$-poset, then:
\begin{itemize}
\item[(a)] if $G=\mathbb{Z}_2$, then $\xind P=\sind \Delta P$;
\item[(b)] if $G\neq \mathbb{Z}_2$, then
\[
\left\lceil\frac{\xind P}{2}\right\rceil\;\le\;\sind\Delta P\;\le\;\xind P.
\]
\end{itemize}
\end{corollary}

\refstepcounter{equation} 
\label{proof: cor main result}
\begin{proof}[Proof of Corollary \ref{cor: main result} assuming Theorem~\ref{Main:Theorem 1}]
The inequality
\[
\sind \Delta P\le \xind P
\]
is exactly \eqref{eq:index_inequality 1}. Thus it remains to prove the reverse bounds. Let \(m=\sind\Delta P\). By Proposition~\ref{proposition:partition-sind}, there is a \(G\)-invariant partition
\[
P=A_0\sqcup A_1\sqcup\cdots\sqcup A_m
\]
such that each \(A_i\) has cross-index zero. Now set \(P_k:=A_0\cup\cdots\cup A_k\). Then \(P_0=A_0\), so \(\xind P_0=0\). If \(G=\mathbb Z_2\), applying Theorem~\ref{Main:Theorem 1}(b) to \(P_k=P_{k-1}\cup A_k\) gives
\[
\xind P_k\le \xind P_{k-1}+\xind A_k+1=\xind P_{k-1} + 1.
\]
Inductively, \(\xind P\le m=\sind\Delta P\). Together with \(\sind\Delta P\le \xind P\), this yields
\[
\xind P=\sind\Delta P.
\]
If \(G\neq \mathbb Z_2\), Theorem~\ref{Main:Theorem 1}(a) gives
\[
\xind P_k\le \xind P_{k-1}+2(\xind A_k+1)=\xind P_{k-1}+2.
\]
Hence, by induction, we obtain
\[
\xind P\le 2m=2\,\sind\Delta P.
\]
Consequently,
\[
\frac{\xind P}{2}\le \sind\Delta P,
\]
and since $\sind\Delta P$ is an integer, it follows that
\[
\left\lceil\frac{\xind P}{2}\right\rceil\;\le\;\sind\Delta P.\qedhere
\]
\end{proof}

Regarding the inequality in part (b) of Corollary~\ref{cor: main result}, we will show (see Remark~\ref{remark: gap between simplicial index & cross index}) that, for a group with at least three elements, the inequality
\[
\sind \Delta P\;\le\;\xind P
\]
can be strict, and moreover, the gap between these two quantities can be made arbitrarily large.

\subsection{Height and dimension bounds}

\label{Subsection: Dimension}

Another fundamental property of the topological index is that for any finite free $G$-simplicial complex $\mathcal{K}$, the topological index is bounded above by its dimension. Here we show that the cross-index also enjoys a similar property. And as an application,  we present an alternative combinatorial proof of the mentioned topological property where the standard proof relies on the vanishing of all homotopy groups of the $n$-sphere up to dimension $n-1$. First, we recall that the \emph{height} of a poset $P$, denoted by $h(P)$, is defined as the maximum length $l$ of any chain $x_0 \prec x_1 \prec \cdots \prec x_l$ in $P$.

\begin{proposition}\label{prop:dimension_bound}
If $P$ is a free $G$-poset of finite height, then its cross-index is bounded above by its height; that is,
\[
\xind(P) \le h(P).
\]
\end{proposition}

\begin{proof}
For each $p \in P$, let $\ell(p)$ denote the length of the longest chain ending at $p$. Consider the $G$-orbits $[p] = \{g\cdot p \mid g \in G\}$, each containing exactly $|G|$ distinct elements since the action is free. We construct a $G$-map $\psi\colon P \to Q_{h(P)}G$ as follows. First, choose a representative $p$ from each orbit $[p]$ and set $\psi(p) = (e, \ell(p))$, where $e$ is the group identity. Then extend $\psi$ equivariantly by defining $\psi(g\cdot p) = (g, \ell(p))$ for all $g \in G$. This map is well-defined because $\ell(p) \le h(P)$ for all $p \in P$, it preserves the $G$-action by construction, and is monotone since $\ell : P\to \mathbb Z$ is evidently monotone.
\end{proof}

\begin{corollary}[{\cite[Proposition 6.2.4]{matouvsek2003using}}]
\label{cor:top_dim_bound}
For any finite free $G$-simplicial complex $\mathcal{K}$,
\[
\ind \mathcal{K} \le \dim(\mathcal{K}).
\]
\end{corollary}

\begin{proof}
The inequality follows from three facts: First, the topological index is bounded by the cross-index, Inequality~\ref{eq:index_inequality 3}, $\ind \mathcal{K} \le \xind \mathcal{F}(\mathcal{K})$. Second, the cross-index is bounded by height, Proposition~\ref{prop:dimension_bound}, $\xind \mathcal{F}(\mathcal{K}) \le h(\mathcal{F}(\mathcal{K}))$. Finally, the height of the face poset equals the complex's dimension, $h(\mathcal{F}(\mathcal{K})) = \dim(\mathcal{K})$. Combining these yields the result.
\end{proof}

\begin{remark}
Proposition~\ref{proposition:chromatic} demonstrates that, in general, this property does not hold for the simplicial index.
\end{remark}

The following rather standard observation helps in reducing questions about possibly infinite posets to questions about finite posets.

\begin{theorem}[Compactness principle for the cross-index]
Let $G$ be a finite group, and let $P$ be a free $G$-poset (not necessarily finite). Then the cross-index of $P$ is determined by its finite $G$-subposets:
\[
\xind(P)=\sup\{\xind(F): F\subseteq P \text{ is a finite $G$-subposet}\}.
\]
\end{theorem}
\begin{proof}
Evidently the right-hand side is at most $\xind P$ from monotonicity of the cross-index. So it remains to show that whenever $\xind P>n$ there exists a finite $G$-subposet $F\subseteq P$ such that $\xind F> n$.

For each finite \(G\)-subposet \(F\subseteq P\), let
\[
M_F:=\{f:P\to Q_nG : f|_F \text{ is a monotone \(G\)-map}\}.
\]
Consider the set \({Q_nG}^{\,P}\) of all maps \(P\to Q_nG\), endowed with the product topology, where \(Q_nG\) carries the discrete topology. Since \(Q_nG\) is finite, it is compact, and thus \(Q_nG^{\,P}\) is compact by Tikhonov's theorem. Moreover, each \(M_F\) is a closed subset of \(Q_nG^{\,P}\), as the requirement that \(f|_F\) be \(G\)-equivariant and order-preserving is determined by finitely many coordinate conditions. Hence \(M_F\) is closed.

The set of all monotone $G$-maps $f: P\to Q_nG$ evidently equals the intersection $\bigcap_{\text{finite}\ F\subseteq P} M_F$. If $\xind P>n$ then this set is empty, and from compactness some finite intersection
\[
M_{F_1}\cap \dots \cap M_{F_N}
\]
is also empty. This implies $\xind F_1\cup\dots\cup F_N > n$, since otherwise any extension of a monotone $G$-map $f : F_1\cup\dots\cup F_N\to Q_nG$ would lie in this intersection. So we have a finite $G$-subposet $F_1\cup\dots\cup F_N$ with cross-index greater than $n$.
\end{proof}

\subsection{Borsuk--Ulam type theorem for cross-index}

\label{Subsection: Borsuk}
The generalized Borsuk--Ulam theorem states that there is no $G$-equivariant continuous map from $E_nG$ to $E_mG$ for $n > m$. In other words, $\ind E_nG = n$. Although the original proof of this theorem relies on sophisticated topological machinery, the analogous result for the cross-index, $\xind Q_nG = n$ for all $n \ge 0$, follows from a simple combinatorial argument. This also follows directly from the generalized Borsuk–Ulam theorem. To see this, observe that the existence of an order-preserving $G$-map from $Q_nG$ to $Q_mG$ for $n > m$ would induce a $G$-equivariant map from $\Delta Q_nG \cong_{G} E_nG$\footnote{In this paper, the notation $\cong_{G}$ denotes isomorphism either in the category of $G$-equivariant topological spaces or in the category of $G$-posets.} to $\Delta  Q_mG \cong_{G} E_mG$, contradicting the generalized Borsuk--Ulam theorem. 

To avoid topological methods as in the above argument, we provide a simple, purely combinatorial argument:

\begin{proposition}\label{prop:xind}
For any $n \ge 0$, we have $\xind Q_nG = n$.
\end{proposition}

\begin{proof}
Suppose, by contradiction, that there exists a $G$-map $\psi \colon Q_nG \to Q_mG$ for some $n > m$. Consider the chain in $Q_nG$:
\[
(e, 0) \preceq (e, 1) \preceq\cdots \preceq (e, n).
\]
Since $\psi$ is order-preserving, we obtain the corresponding chain in $Q_mG$:
\[
\psi((e, 0)) \preceq \psi((e, 1)) \preceq \cdots \preceq \psi((e, n)).
\]
By the pigeonhole principle (as $n > m$), there exist indices $a > b$ in $\{1, \ldots, n+1\}$ such that $\psi((e, a)) = \psi((e, b))$. Now, take any non-trivial element $g \in G$ with $g \neq e$. We have:
\[
\psi((g, a)) = g \cdot \psi((e, a)) = g \cdot \psi((e, b)).
\]
However, in $Q_nG$, we have the relation $(g, a) \succ (e, b)$, while in $Q_mG$, the elements $\psi((g, a)) = g \cdot \psi((e, b))$ and $\psi((e, b))$ are incomparable. This contradicts the assumption that $\psi$ is order-preserving.
\end{proof}
However, we do not claim that the topological version can be deduced directly from this simple combinatorial analogue. For a detailed comparison of the complexity between the Borsuk--Ulam theorem and its various topological and combinatorial counterparts, we refer the reader to the recent survey~\cite{daneshpajouh2025box}.

\section{Proofs of the Main Results}

\subsection{Joins and comparison between topological index and cross-index}

\label{Subsection: Join}
The topological index is sub-additive under the join operation, i.e., if \( X \) and \( Y \) are two topological \( G \)-spaces, then we have
\begin{equation}\label{eq: join}
    \ind X \ast Y \le \ind X + \ind Y + 1.
\end{equation}

This follows from the simple observation that given \( G \)-maps \( X \to E_nG \) and \( Y \to E_mG \), they induce a \( G \)-map (join of maps) \( X \ast Y \to E_nG \ast E_mG = E_{n+m+1}G \). However, this inequality can be strict (see~\cite{csorba2007homotopy}). Nevertheless, we show that equality always holds for the cross-index.  To proceed, we first recall the definition of the join of posets.

\begin{definition}
The join of two disjoint posets $P$ and $Q$, denoted $P \ast Q$, is the poset whose underlying set is the disjoint union $P \cup Q$ and whose order relation preserves the original orders within $P$ and within $Q$, while additionally declaring that every element of $P$ is less than every element of $Q$. If $P$ and $Q$ are $G$-posets, then $P \ast Q$ inherits a natural $G$-action from $P$ and $Q$, making it a $G$-poset.    
\end{definition}

\begin{proposition}\label{prop: join}
For free $G$-posets $P$ and $Q$ we have
$$\xind(P\ast Q)= \xind P+ \xind Q +1.$$
\end{proposition}

\begin{proof}
Let $\xind P = m$ and $\xind Q = n$. So there are $G$-maps from $P$ and $Q$ to $Q_mG$ and $Q_nG$ respectively. These maps induce a $G$-map from $P\ast Q$ to $Q_m\ast Q_n\cong_{G} Q_{n+m+1}$. Thus $\xind P\ast Q\le \xind P+ \xind Q +1$.  Conversely, suppose that $\xind(P\ast Q) < \xind P+ \xind Q +1$. So there are $t\le m+n$  and a $G$-map
$\psi : P\ast Q\to Q_t$. By the definition of the join and $\psi $, for each $p\in P$ and  $q\in Q$ we have $\psi (p) \prec \psi(q)$. Indeed, $\psi (p) \preceq \psi(q)$ as $\psi$ is order preserving and $p\prec q$. Moreover, $\psi (p) \neq \psi(q)$, since otherwise from one hand we had
$\psi(gp)=g\psi(p)=g\psi(q)$ for every $g$. On the other hand, the inequality $gp \prec q$ (as every element of $P$ is strictly less than every element of Q in $P\ast Q$) and the fact that $\psi$ is an order preserving map, would imply that $g\psi(q)\preceq \psi(q)$. This is a contradiction as such elements are not comparable in $Q_tG$ when $g\ne e$. Thus,  $\psi (p) \prec \psi(q)$. Now, by restriction $\psi$ to $P$ and $Q$ we can find $G$-maps from $P$ and $Q$ to $Q_a$ and $Q_b$ respectively for some $a,b$ such that $a+b\le t-1$. Thus $a+b < m+n$, which implies $a < m$ or $b < n$. This contradicts our assumption.
\end{proof}

We are now prepared to prove Theorem~\ref{thm:large-gap}. As a first step, we construct a $G$-poset with cross-index 2 and topological index 1, as described in the following example.
\begin{example}\label{ex:21}
Let $G$ be an arbitrary group with at least two elements, and $g_*$ be a non-trivial element of $G$, i.e., $g_*\neq e$. Let  $P_G$ be the $G$-poset with $6|G|$ elements whose elements are
\[P_G=\{(g, a_i) : 1\le i\le 6, g\in G \}\]
and whose ordering is given by
\begin{itemize}
    \item $(g, a_i) \preceq  (g, a_{i+2})$ for every $g\in G, i=1,2$,
    \item $(g, a_1) \preceq  (gg_{*}, a_{4})$ for every $g\in G$,
    \item $(g, a_2) \preceq  (g, a_{3})$ for every $g\in G$,
    \item $(g, a_i) \preceq  (g', a_{i+2})$ for every $g, g'\in G, i=3,4$.
\end{itemize}
The group $G$ acts naturally on $P$ by acting on the first component, i.e., $h\cdot(g, a_i)=(hg, a_i)$.
\end{example}
This example is inspired by a similar construction in \cite[Theorem 9]{Sim}, which treats the case of $\mathbb{Z}_2$. 
We illustrate the construction for the case $G=\mathbb{Z}_2=\{\pm 1\}$. In this case, there is only one choice for $g_{*}$, i.e., $g_{*}=-1$, and $P_{\mathbb{Z}_2}$ is the $\mathbb{Z}_2$-poset whose Hasse diagram is depicted in Fig. \ref{fig:ex-with-z2}. 

\begin{figure}[h]
\centering
\begin{tikzpicture}[
    node distance=1.5cm,  
    every node/.style={
        draw, 
        circle, 
        minimum size=5mm,  
        inner sep=1pt,    
        font=\small        
    }
]
    \node (a1) {$(+,a_1)$};
    \node[right=of a1] (a2) {$(+,a_2)$};
    \node[right=of a2] (a3) {$(-,a_1)$};
    \node[right=of a3] (a4) {$(-,a_2)$};
    
    \node[above=of a1] (b1) {$(+,a_3)$};
    \node[right=of b1] (b2) {$(+,a_4)$};
    \node[right=of b2] (b3) {$(-,a_3)$};
    \node[right=of b3] (b4) {$(-,a_4)$};
    
    \node[above=of b1] (c1) {$(+,a_5)$};
    \node[right=of c1] (c2) {$(+,a_6)$};
    \node[right=of c2] (c3) {$(-,a_5)$};
    \node[right=of c3] (c4) {$(-,a_6)$};
    
    \draw[->, line width=1.2pt] (a1) -- (b1);
    \draw[->, line width=1.2pt] (a1) -- (b4);
    \draw[->, line width=1.2pt] (a2) -- (b2);
    \draw[->, line width=1.2pt] (a2) -- (b1);
    \draw[->, line width=1.2pt] (a3) -- (b3);
    \draw[->, line width=1.2pt] (a3) -- (b2);
    \draw[->, line width=1.2pt] (a4) -- (b3);
    \draw[->, line width=1.2pt] (a4) -- (b4);
    
    \draw[->, line width=1.2pt] (b1) -- (c1);
    \draw[->, line width=1.2pt] (b1) -- (c3);
    \draw[->, line width=1.2pt] (b2) -- (c2);
    \draw[->, line width=1.2pt] (b2) -- (c4);
    \draw[->, line width=1.2pt] (b3) -- (c1);
    \draw[->, line width=1.2pt] (b3) -- (c3);
    \draw[->, line width=1.2pt] (b4) -- (c4);
    \draw[->, line width=1.2pt] (b4) -- (c2);
\end{tikzpicture}
\caption{Hasse diagram of $P_{\mathbb{Z}_2}$}
\label{fig:ex-with-z2}
\end{figure}
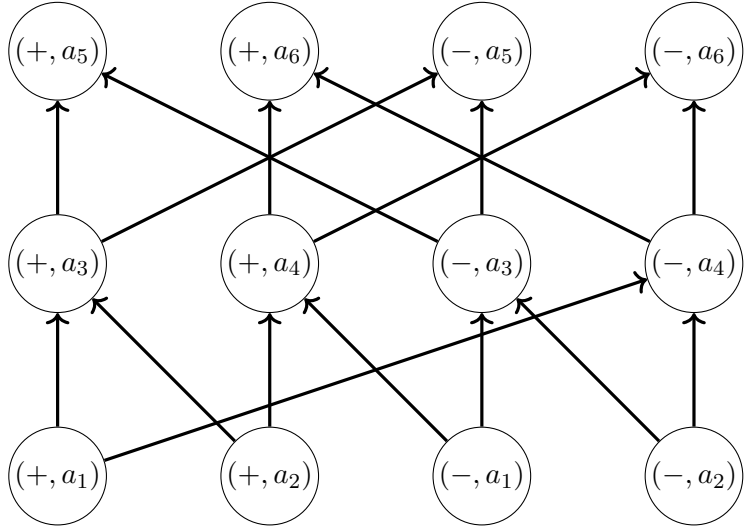
The order complex $\Delta P_{\mathbb{Z}_2}$ of $P_{\mathbb{Z}_2}$ is depicted in Fig. \ref{fig:order-complex}.

\begin{figure}[h!]
    \centering
    \includegraphics[scale=0.50]{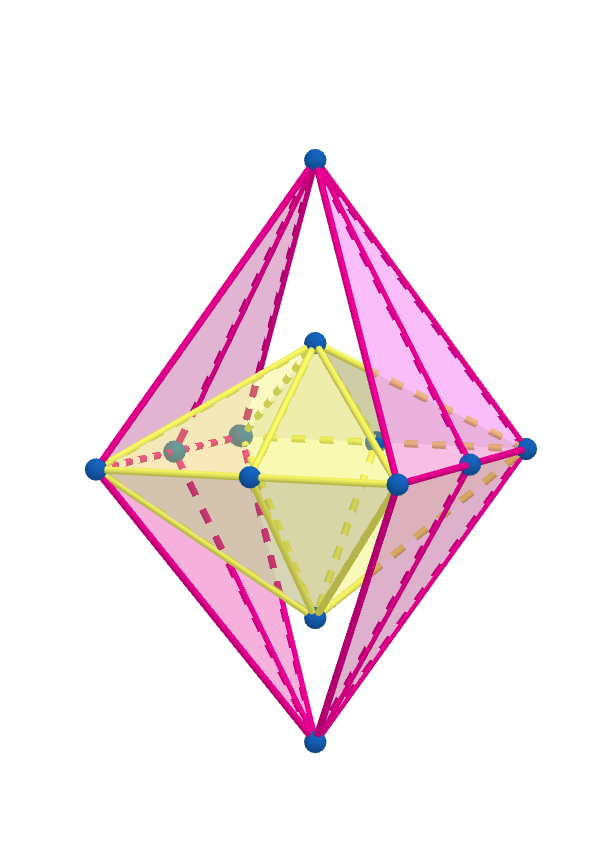}
    \\
    \centering
    \includegraphics[width=.7\textwidth]{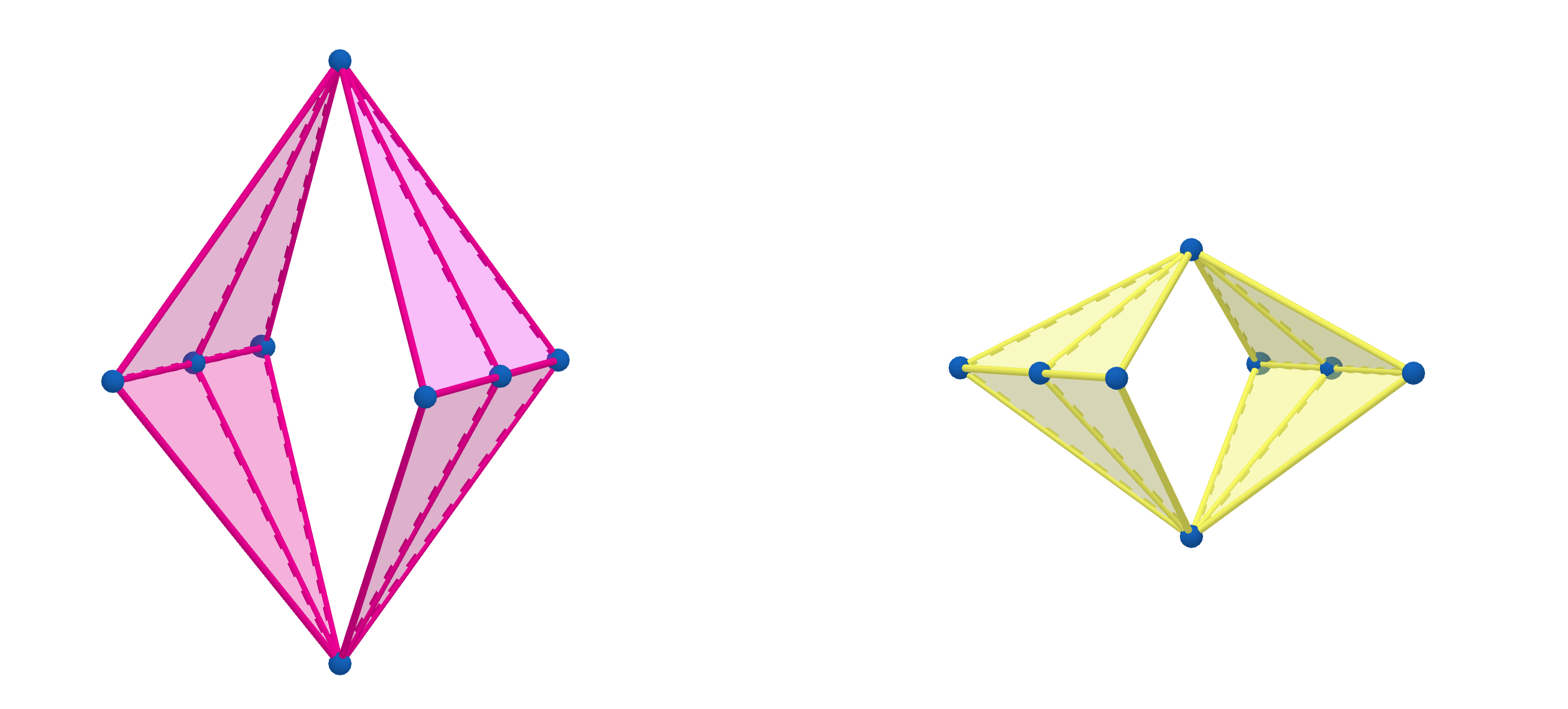}
    \caption{The order complex of $P_{\mathbb{Z}_2}$ appears at the top. It is the union of two subcomplexes (shown in pink and yellow). After rotating the yellow subcomplex by about \(90^\circ\) around the \(z\)-axis and placing it appropriately, the two subcomplexes assemble into the full order complex.}
    \label{fig:order-complex}
\end{figure}

It is $\mathbb{Z}_2$-homotopy equivalent to the $1$-dimensional $\mathbb{Z}_2$-space in Fig. \ref{fig:eq-1-dim} and therefore its topological-index is one.

\begin{figure}[h!]
    \centering
    \includegraphics[scale=0.50]{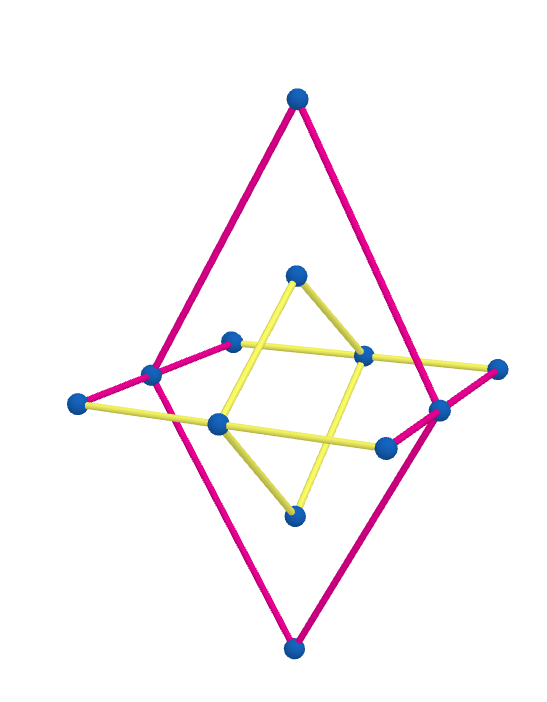}
    \caption{$\mathbb{Z}_2$-Homotopy equivalence of $P_{\mathbb{Z}_2}$ to the 1-dimensional space shown.}
    \label{fig:eq-1-dim}
\end{figure}

Now, we prove the cross and topological indices of Example \ref{ex:21} in the following next two lemmas.
\begin{lemma}\label{lem: Xind P_G=2}
    $\xind P_G = 2$.
\end{lemma}
\begin{proof}
The inequality $\xind P_G\le 2$ comes from the fact that the cross-index of a $G$-poset is always bounded from above by its height, see Proposition \ref{prop:dimension_bound}. For the other side, assume
$\psi: P_G\to Q_nG$ is a $G$-map. We must show that $n\ge 2$. Consider the following $G$-subposet of $P_G$:
$$P_1=\{(g, a_i) : 1\le i\le 4, g\in G\}.$$
The cross-index of $P_1$ cannot be zero as there is a path between two elements from the same orbit in the comparability graph of $P_1$ \cite[Proposition 2]{bui2023topological}. That is,

$$(g_{*}, a_1) \prec (g_{*}, a_3) \succ (g_{*}, a_2) \prec (g_{*}, a_4) \succ (e, a_1).$$
Thus, $\xind P_1\ge 1$, and therefore one of the maximal elements of $P_1$, say for simplicity $(e, a_3)$, must be labeled with $(g, m)$ for some $g\in G$ where $m\ge 1$: $\psi((e, a_3))= (g,m)$. But, we have $(e, a_3) \prec (e, a_5)$ and $(e, a_3) \prec (g_{*}, a_5)$, which implies $\psi ((e, a_3)) \preceq \psi ((e, a_5))$ and $\psi ((e, a_3)) \preceq \psi ((g_{*}, a_5))=g_{*}\psi ((e, a_5))$. Thus, the value of $\psi ((e, a_3))$ (without considering its sign) must be strictly less than the value of $\psi ((e, a_5))$ as for any $1\le i\le n+1$ the elements $(g, i)$ and $(g',i)$ are not comparable in $Q_nG$ where $g\neq g'$. Therefore, $\xind P_G\ge 2$, and therefore combining this result with the previous part implies $\xind P_G=2$. 
\end{proof}

\begin{lemma}\label{lem: ind P_G=1}
$\ind \Delta P_G = 1$.
\end{lemma}

\begin{proof}
Observe that every triangle in $\Delta P$ has a free edge. Indeed, for a triangle $(g_1, a_{i_1}) \prec (g_2, a_{i_2}) \prec (g_3, a_{i_3})$, the edge $\{(g_1, a_{i_1}), (g_3, a_{i_3})\}$ is free. Equivalently, every 2-simplex of $\Delta P$ contains a ``long edge'' (joining its minimal and maximal vertices) that lies in exactly one triangle. To see this, note that the elements of $P$ split into three levels:
\begin{itemize}
    \item \(L_1 = \{(g, a_1), (g, a_2):\, g\in G\,\}\), 
    \item \(L_2 = \{(g, a_3), (g, a_4):\, g\in G\,\}\), 
    \item \(L_3 = \{(g, a_5), (g, a_6):\, g\in G\,\}\).
\end{itemize}

The only nontrivial comparabilities are:
\begin{align*}
& (g, a_1) \prec (g, a_3), \quad (g, a_1) \prec (gg_*, a_4);\\
& (g, a_2) \prec (g, a_3), \quad (g, a_2) \prec (g, a_4);
\end{align*}
and for all $h$,
\[
(g, a_3) \prec (h, a_5), \qquad (g, a_4) \prec (h, a_6).
\]
In particular, there is no relation between level-2 and the ``wrong'' top:
\[
(g, a_3) \not\prec (h, a_6), \qquad (g, a_4) \not\prec (h, a_5),
\]
and all elements in the same level are incomparable as well.

A 2-simplex in \(\Delta P\) is a chain \(x\prec y\prec z\) with \(x\in L_1\), \(y\in L_2\), \(z\in L_3\).
The only possibilities are:
\begin{itemize}
    \item [(i)] \((g, a_1) \prec (g, a_3) \prec (h, a_5)\);
    \item [(ii)] \((g, a_2) \prec (g, a_3) \prec (h, a_5)\);
    \item [(iii)] \((g, a_1) \prec (gg_*, a_4) \prec (h, a_6)\);
    \item [(iv)] \((g, a_2) \prec (g, a_4) \prec (h, a_6)\).
\end{itemize}

In any poset, an edge \(\{x,z\}\) in the order complex lies in exactly \(|(x,z)|\) many 2-simplices, where \((x,z)\) denotes the open interval \(\{y : x\prec y\prec z\}\). Now compute these open intervals for the four types above. In each case, any \(y\) with \(x\prec y\prec z\) must lie in \(L_2\).
From the comparabilities: above \((g, a_1)\) in \(L_2\) are exactly \((g, a_3)\) and \((gg_*, a_4)\), but among these only \((g, a_3)\) lies below any \((h, a_5)\) and only \((gg_*, a_4)\) lies below any \((h, a_6)\); above \((g, a_2)\) in \(L_2\) are exactly \((g, a_3)\) and \((g, a_4)\), but among these only \((g, a_3)\) lies below any \((h, a_5)\) and only \((g, a_4)\) lies below any \((h, a_6)\).
Therefore, in each of (i)–(iv) we have \((x,z) = \{y\}\); the open interval is a singleton. 

Consequently, for every triangle \(x\prec y \prec z\), the long edge \(\{x,z\}\) is contained in exactly one 2‑simplex, hence is a free edge. Thus every triangle in \(\Delta P\) has a free edge. Therefore, we can equivariantly deform $\Delta P$ to a one dimensional equivariant subspace. Thus, $\ind \Delta P\le 1$ as the topological index is always bounded above by its dimension and it is preserved under the equivariant homotopy. Indeed, the topological index is exactly one as it is known that the topological index is zero if and only if the cross-index is zero~\cite[Proposition 1]{bui2023topological}.
\end{proof}

\begin{proof}[Proof of Theorem \ref{thm:large-gap}]
Let $P_G$ be the $G$-poset in Example \ref{ex:21}. By Lemma \ref{lem: ind P_G=1}, we have $\ind \Delta P_G = 1$ and by Lemma~\ref{lem: Xind P_G=2} $\xind P_G=2$. It follows that the cross-index of the poset $Q={P_G}^{\ast n}$ is $3n-1$ while its topological-index is at most $2n-1$. Indeed, by the additivity of the cross-index on the join operation, Proposition~\ref{prop: join}, we have
$$\xind Q= n\xind P+(n-1)=2n+n-1=3n-1.$$
On the other hand by the sub-additivity of the topological-index on the join operation, Inequality~\eqref{eq: join}, we get
$$\ind \Delta Q=\ind \underbrace{\Delta P_G\ast\cdots\ast\Delta P_G}_{n}\le n\ind \Delta P_G+n-1=n+n-1=2n-1.$$
For the lower bound, since $\Delta P_G$ is connected, there is a $G$-equivariant map $\psi \colon E_1G \to \Delta P_G$. to see this, first fix $|G| + 1$ distinct elements $\{x_e\} \cup \{y_g\}_{g \in G}$ in $P_G$. For each $g \in G$, let $\phi_{e,g}$ be a path in $\Delta P$ connecting $x_e$ to $y_g$. The map $\psi$ is then defined by $\psi([g, h, t]) = g\phi_{e,g^{-1}h}(t)$ for all $g,h \in G$ and $t \in [0,1]$ is the desired equivarian map. This construction extends to a $G$-equivariant map from 
$$E_{2n-1}G = \underbrace{(G \ast G)\ast\cdots\ast(G\ast G)}_{n}\to \underbrace{\Delta P\ast\cdots \ast\Delta P}_{n} \cong_{G}\Delta Q.$$

Finally, monotonicity of the topological index yields:
    \[
    2n-1 = \ind E_{2n-1}G \le \ind \Delta Q.
    \]
    Thus, we have $\ind \Delta Q = 2n-1$ which completes the proof of our claim.
\end{proof}

\begin{remark}
    Later, in Remark~\ref{remark: final}, we will present a distinct example when the group has at least three elements.
\end{remark}

It is now natural to pose the following question.

\begin{question}
\label{que: que3}
Let \( n > 1 \) be a positive integer. Does there exist a connected\footnote{A poset is called connected if its order complex is.} finite free \( G \)-poset \( P \) such that
\[
\ind \Delta P = 1 \quad \text{but} \quad \xind P = n?
\]
\end{question}

An affirmative answer to Question~\ref{que: que3} for all $n \ge 2$ would provide a complete solution to Question~\ref{que: que2}.
\begin{proposition}
    Let $m,n\in \mathbb N$ so that $2 \le m < n$. Suppose that there exists a connected $G$-poset $P$ satisfying that $\xind P = n - m + 1$ and $\ind\Delta P = 1$. Then $Q = P * Q_{m-2}G$ satisfies $\ind\Delta Q = m$ and $\xind Q = n$.
\end{proposition}
\begin{proof}
 Using the additivity of cross-index under joins (Proposition~\ref{prop: join}), we obtain
    \[
    \xind Q = \xind P + \xind Q_{m-2}G + 1 = (n - m + 1) + (m - 2) + 1 = n.
    \]
    To verify $\ind \Delta Q = m$, we first note the upper bound comes from the subadditivity of join (Inequality~\ref{eq: join}):
    \[
    \ind \Delta Q \le \ind \Delta P + \ind \Delta Q_{m-2}G + 1 = 1 + (m - 2) + 1 = m.
    \]
For the lower bound, because $\Delta P$ is connected, there exists a $G$-equivariant map $E_1G \to \Delta P$. This map extends to a $G$-equivariant map
$$E_mG = (G \ast G) \ast E_{m-2}G \longrightarrow \Delta P \ast \Delta Q_{m-2}G \cong_{G} \Delta Q.$$
Finally, monotonicity of the topological index yields:
    \[
    m = \ind E_mG \le \ind \Delta Q.
    \]
Thus, we have $\ind \Delta Q = m$ which completes the proof of our claim.
\end{proof}

\begin{remark}
Regarding Question~\ref{que: que3}, the answer is affirmative in the case where $G$ is a nontrivial (discrete) free group. Indeed, by Proposition~\ref{prop:xind} we have 
\(
\xind Q_n G = n,
\)
while it is known that 
\(
\ind \Delta(Q_n G) \le 1
\). More generally, if $G$ is a nontrivial free (discrete) group, then for every free $G$-space $X$, one has $\ind_G X \le 1$. Indeed, if $G$ is a non-trivial free group on a set $S$, then a standard model for the classifying space $BG$ is the rose $R_S$, i.e., a wedge of $|S|$ circles. Since $R_S$ is a $K(G,1)$, its universal cover $\widetilde{R_S}$ is contractible and admits a free $G$-action by deck transformations. The CW structure lifts to $\widetilde{R_S}$, making it a $1$-dimensional $G$-CW complex, namely the Cayley graph of $G$ with respect to $S$. Hence $\widetilde{R_S}$ serves as a $1$-dimensional model for the total space $EG$. For further details, see~\cite[Section 1.B]{Hatcher2002AlgebraicTopology}. Therefore, $EG$ has a $1$-dimensional free $G$-CW model $E_1G$, and equivariant obstruction theory guarantees the existence of a $G$-equivariant map\footnote{The obstruction theory here is rather trivial: The vertices of $EG$ can be $G$-equivariantly mapped to $E_1G=G*G$ arbitrarily. Then any edge $s$ representing an orbit $Gs$ may be mapped to $G*G$ using its connectedness, this map may be extended to $Gs$ by $G$-equivariance, and this may be repeated for any orbit of edges until the whole $EG$ is mapped.}
\[
EG \longrightarrow E_1G.
\]
Composing a given classifying map $X \to EG$ with this map yields a $G$-map $X \to E_1G$, and hence $\ind_G(X) \le 1$. For completeness, we describe the situation explicitly in the particular case $G = \mathbb{Z}$. In this case, a standard model for $EG$ is $\mathbb{R}$, with $\mathbb{Z}$ acting by integer translations. We now construct an explicit $\mathbb{Z}$-equivariant map
\[
f \colon \mathbb{R} \longrightarrow \mathbb{Z} * \mathbb{Z},
\]
For $x \in \mathbb{R}$, write $x = n + y$ with $n = \lfloor x \rfloor \in \mathbb{Z}$ and $y \in [0,1)$. Define
\[
f(x) =
\begin{cases}
[n, n, 2y], & 0 \le y \le \frac{1}{2},\\[4pt]
[n+1, n, 2 - 2y], & \frac{1}{2} \le y < 1.
\end{cases}
\]
At $y=0$ we obtain $f(n) = [n,n,0]$, a vertex in the first copy $\mathbb{Z}$, and at $y = \tfrac{1}{2}$ both formulas yield $[n,n,1]$, a vertex in the second copy $\mathbb{Z}$. As $y \to 1^{-}$, we have $f(n+y) \to [n+1,n,0]$, which matches continuously with the definition on the next interval.

To verify equivariance, let $m \in \mathbb{Z}$ and write $x = n + y$ as above. Then
\[
f(x+m) = f((n+m)+y) =
\begin{cases}
[n+m, n+m, 2y], & 0 \le y \le \frac{1}{2},\\[4pt]
[n+m+1, n+m, 2 - 2y], & \frac{1}{2} \le y < 1,
\end{cases}
\]
which is precisely $m \cdot f(x)$ under the diagonal action. Thus $f$ is a continuous $\mathbb{Z}$-equivariant map
\[
f \colon \mathbb{R} \to \mathbb{Z} * \mathbb{Z},
\]
given explicitly in the join coordinates, and this realizes the bound $\ind_{\mathbb{Z}}(\mathbb{R}) \le 1$ in this concrete model.
\end{remark}

\subsection{Cross-index of a union of zero index sets}

\label{subsection: union}
Now we are ready to examine the union property for cross-index which notably presents a behavior somewhat distinct from its topological counterpart. As stated in the introduction, the topological index enjoys the union property: for any free $G$-simplicial complex $X$ with $G$-equivariant subcomplexes $\mathcal{K}$ and $\mathcal{L}$ so that $X = \mathcal{K} \cup \mathcal{L}$, we have
\[
\ind X \le \ind \mathcal{K} + \ind \mathcal{L} + 1.
\]

However, before addressing the general case for the cross-index, we examine why a similar property holds for the face poset of simplicial complexes (beginning with the proof of Theorem~\ref{Thm: union1}) and as a corollary, we demonstrate how the topological version can be derived from this. This discussion requires the following preliminary definition and a proposition.

\begin{definition}
A subset $A$ of a poset $(P, \preceq)$ is called downward-closed if, whenever $y \in A$ and $x \preceq y$ for some $x \in P$, then $x \in A$.
\end{definition}

\begin{lemma}
\label{lem: downward-closed poset}
 Let $P=A\cup B$ be a free \(G\)-poset, where \(A\) and \(B\) are \(G\)-invariant subposets of $P$. If $A$ is downward-closed, then
\[
\xind P\le \xind A + \xind B + 1.
\]   
\end{lemma}

\begin{proof}
Without loss of generality, we may assume that $A\cap B=\emptyset$. Indeed, the general case $P=A\cup B$ can be reduced to the situation in which $B=P\setminus
A$. Since $P\setminus A\subset B$ and the cross-index is monotone non-increasing
under the passage to subposets, we have
\[
\xind(P\setminus A)\le \xind B.
\]
Therefore, it suffices to work under the assumption $A\cap B=\emptyset$. Moreover, in this setting, for any $x\in A$ and $y\in B$ it is impossible to have $y \prec x$, because in that case $y$
would belong to $A$, as $A$ is downward-closed. Given a pair of $G$-order preserving maps $A\to Q_nG$ and $B\to Q_mG$, we construct their union
\[
f:P\to Q_nG*Q_mG\cong_{G} Q_{n+m+1}G.
\]
The map $f$ is order-preserving when restricted to $A$ and to $B$ separately, so it
remains to verify monotonicity for pairs $(x,y)$ with $x\in A$ and $y\in B$. For such a pair, either $x$
and $y$ are incomparable, or $x \prec y$; in the latter case, by construction we have $f(x)\prec f(y)$. Consequently,
$f$ is a $G$-order-preserving, which implies
\[
\xind P\le \xind A+\xind B+1.\qedhere
\]
\end{proof}

\begin{proof}[Proof of Theorem \ref{Thm: union1}]
The theorem follows from Lemma~\ref{lem: downward-closed poset} together with the fact that any subcomplex of the face poset of a simplicial complex is downward-closed.
\end{proof}

As an immediate consequence of this combinatorial result, we recover the topological version of the union property.

\begin{corollary}{\cite{schw1966}}
Let $X = \mathcal{K} \cup \mathcal{L}$ be a free $G$-simplicial complex where $\mathcal{K}$ and $\mathcal{L}$ are its $G$-invariant subcomplexes. Then
\[
\ind X \le \ind \mathcal{K} + \ind \mathcal{L} + 1.
\]
\end{corollary}

\begin{proof}
By Proposition~\ref{pro:Approximation}, there exists $r \ge 0$ such that 
$\ind X = \xind \mathcal{F}(\sd^r(X))$, $\ind \mathcal{K} = \xind \mathcal{F}(\sd^r(\mathcal{K}))$,
and $\ind \mathcal{L} = \xind \mathcal{F}(\sd^r(\mathcal{L}))$.
The result then follows from Theorem~\ref{Thm: union1}.
\end{proof}

Now we present the proof of the most surprising result of the paper, Theorem~\ref{Main:Theorem 1}. Before outlining the strategy of the proof, we introduce a piece of notation that will be used throughout.

\begin{definition}
For a group \(G\), define \(\mu(G)\) to be the supremum of the cross-index taken over all \(G\)-posets \(P\) that can be written as the union of two \(G\)-subposets of cross-index zero, that is,
\[
\mu(G) = \sup\{\xind P : P = A \cup B,\ \xind A = \xind B = 0\}.
\]
\end{definition}

Our strategy is as follows. First, we prove that \(\mu(G)\) is always bounded above by \(2\). Next, we adapt the argument used in the special case \(G = \mathbb{Z}_2\) to show that, in this situation, the upper bound can be improved from \(2\) to \(1\). Then, to establish the general inequalities in Theorem~\ref{Main:Theorem 1}, we use an induction argument to reduce the general case to the situation where \(\xind A=\xind B=0\).

\begin{theorem}\label{thm:bound <=2}
    For any non-trivial group $G$, we have $\mu(G)\le 2$.
\end{theorem}
\begin{proof}
    Let $P=A\cup B$ for 
    the two $G$-posets $A,B$ with $\xind A=\xind B=0$ and the two $G$-maps $\psi_A,\psi_B$. We show that $\xind P\le 2$ by constructing a $G$-map $\psi$. First, observe that we may assume without loss of generality that $A$ and $B$ are disjoint. Indeed, if they are not disjoint, then we can replace $B$ with $B \setminus A$. Now, let $B_A$ be the set of those $b\in B$ that have $a',a''\in A$ such that $a'\prec b\prec a''$. Note that from transitivity and the property of $\psi_A$ it follows that $\psi_A(a')=\psi_A(a'')$. Hence we may correctly (independently on the choice of $a'$ and $a''$) extend $\psi_A(b)=\psi_A(a')=\psi_A(a'')$. This extension is also equivariant, since $a'\prec b\prec a''$ implies $ga'\prec gb\prec ga''$. The extended $\psi_A : A\sqcup B_A\to G$ still maps any edge of the graph of $P$ to a single point. Indeed, if $b_1\prec b_2\in B_A$, then this chain can be extended to $a_1'\prec b_1\prec b_2\prec a_2''$, implying $a_1'\prec a_2''$, implying $\psi_A(a_1')=\psi_A(b_1)=\psi_A(b_2)=\psi_A(a_2'')$. If $A\ni a'\prec b\in B_A$ then $\psi_A(a')=\psi_A(b)$ from the correctness of the extension of $\psi_A$. After merging $B_A$ to $A$ (using the extended $\psi_A$) and removing $B_A$ from $B$ the remaining subset $B$ partitions in two parts: $D(B)$ containing elements of $B$ not comparable to $A$ or below some elements of $A$ and $U(B)$ containing elements of $B$ above some elements of $A$.

    We construct a map $\psi$ as follows:
    \begin{itemize}
        \item if $x\in D(B)$ then $\psi(x) = (g,0)$ if $\psi_B(x) = (g,0)$.
        \item if $x\in A$ then $\psi(x) = (g,1)$ if $\psi_A(x) = (g,0)$.
        \item if $x\in U(B)$ then $\psi(x) = (g,2)$ if $\psi_B(x) = (g,0)$.
    \end{itemize}

    We can easily check that this is a $G$-map. In order to check its monotonicity one essentially has to consider $x\prec y$ in different parts of $B$. Note that $x\in U(B)$ and $y\in D(B)$ is impossible since it would imply some element of $A$ is below $y$, contradicting the definition of $D(B)$. So the only possibility is $x\in D(B)$ and $y\in U(B)$, which does not violate the monotonicity.
\end{proof}

\begin{theorem}
\label{thm: mu(Z_2)=1}
    $\mu(\mathbb{Z}_2)=1$.
\end{theorem}

\begin{proof}
We continue the previous proof with $G=\mathbb Z_2$. Denote $D(B)$ and $U(B)$ by $D$ and $U$, since we are going to use subscripts with them. Denote the elements of $G$ by $\{+1, -1\}$. The $G$-map $D\sqcup U\to G$ means a splitting $D=D_+\sqcup D_-$ and $U=U_+\sqcup U_-$ so that $D_+$ is not connected to either $D_-$ or $U_-$ in the graph of $P$, and $D_-$ is not connected to either $D_+$ or $U_+$ in the graph of $P$, and
    \[
    -1\cdot D_+=D_-, \quad -1\cdot D_-=D_+, \quad -1\cdot U_+=U_-, \quad -1\cdot U_-=U_+.
    \]
    
    Let us first split $A=A_\emptyset\sqcup A_+\sqcup A_-\sqcup A_\pm$, where $A_\emptyset$ consists of elements not above any element of $D$, $A_+$ consists of elements above some element of $D_+$ and not above any element of $D_-$, $A_-$ consists of elements above some element of $D_-$ and not above any element of $D_+$, $A_\pm$ consists of elements that are above some element of $D_+$ and above some element of $D_-$. 
    
    From transitivity it follows that the elements of $A_\emptyset$ cannot be above any element of $A_+\sqcup A_-\sqcup A_\pm$, $A_+$ cannot be above any element of $A_-\sqcup A_\pm$, $A_-$ cannot be above any element of $A_+\sqcup A_\pm$. Moreover, from transitivity no element of $A_\pm$ is below any element of $U$.   
    
    For the group action,
    \[
    -1\cdot A_\emptyset= A_\emptyset, \quad -1\cdot A_\pm= A_\pm, \quad -1\cdot A_+=A_-,\quad -1\cdot A_-=A_+.
    \]
    
    Now we invoke the $G$-equivariant map $A\to G$ that splits $A=A^+\sqcup A^-$ so that $-1\cdot A^+=A^-$ and $-1\cdot A^+=A^-$ and also splits
    \[
    A_\emptyset = A_\emptyset^+\sqcup A_\emptyset^-,\quad 
    A_+ = A_+^+\sqcup A_+^-,\quad A_- = A_-^+\sqcup A_-^-,\quad A_\pm = A_\pm^+\sqcup A_\pm^-.
    \]
    All these sets are send to each other by the action of $G$.
    
    Now set
    \[
    X_+ = D_+\sqcup A_+^+\sqcup A_\emptyset^+,\quad X_-=D_-\sqcup A_-^-\sqcup A_\emptyset^-
    \]
    and
    \[
    Y_+ = U_+\sqcup A_+^-\sqcup A_\pm^-,\quad Y_-=U_-\sqcup A_-^+\sqcup A_\pm^+.
    \]
    For the group action,
    \[
    -1\cdot X_-=X_+, \quad -1\cdot X_+=X_-,\quad -1\cdot Y_- = Y_+, \quad -1\cdot Y_+ = Y_-.
    \]

To elucidate the relations among these 12 sets defined in the proof, the Hasse diagram in Figure~\ref{fig:mu(Z_2)=1} provides a compact visual representation.

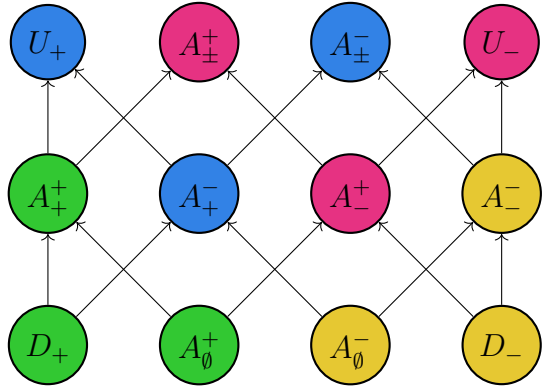
\begin{figure}[h!]
    \centering
\begin{tikzpicture}[
    node style/.style={circle, draw=black, thick, font=\sffamily\bfseries, minimum size=0.8cm, text depth=0pt, align=center}
]

    \definecolor{myblue}{RGB}{50, 130, 230}
    \definecolor{mypink}{RGB}{230, 50, 130}
    \definecolor{mygreen}{RGB}{50, 200, 50}
    \definecolor{myyellow}{RGB}{230, 200, 50}

    \node[node style, fill=myblue] (U_plus) at (0, 6) {\(U_+\)};
    \node[node style, fill=mypink] (A_bar_plus) at (2, 6) {\(A_{\pm}^{+}\)};
    \node[node style, fill=myblue] (A_bar_minus) at (4, 6) {\(A_{\pm}^{-}\)};
    \node[node style, fill=mypink] (U_minus) at (6, 6) {\(U_{-}\)};

    \node[node style, fill=mygreen] (A_plus_green) at (0, 4) {\(A_+^+\)};
    \node[node style, fill=myblue] (A_plus_blue) at (2, 4) {\(A_+^{-}\)};
    \node[node style, fill=mypink] (A_bar_plus_mid) at (4, 4) {\(A_{-}^{+}\)};
    \node[node style, fill=myyellow] (A_minus_yellow) at (6, 4) {\(A_{-}^-\)};

    \node[node style, fill=mygreen] (D_plus) at (0, 2) {\(D_+\)};
    \node[node style, fill=mygreen] (A_empty_green) at (2, 2) {\(A_{\emptyset}^{+}\)};
    \node[node style, fill=myyellow] (A_empty_yellow) at (4, 2) {\(A_{\emptyset}^{-}\)};
    \node[node style, fill=myyellow] (D_minus) at (6, 2) {\(D_-\)};

\draw[->] (A_plus_green) -- (U_plus);
\draw[->] (A_plus_blue) -- (U_plus);

\draw[->] (A_plus_green) -- (A_bar_plus);
\draw[->] (A_bar_plus_mid) -- (A_bar_plus);

\draw[->] (A_plus_blue) -- (A_bar_minus);
\draw[->] (A_minus_yellow) -- (A_bar_minus);

\draw[->] (A_bar_plus_mid) -- (U_minus);
\draw[->] (A_minus_yellow) -- (U_minus);

\draw[->] (D_plus) -- (A_plus_green);
\draw[->] (A_empty_green) -- (A_plus_green);

\draw[->] (D_plus) -- (A_plus_blue);
\draw[->] (A_empty_yellow) -- (A_plus_blue);

\draw[->] (A_empty_green) -- (A_bar_plus_mid);
\draw[->] (D_minus) -- (A_bar_plus_mid);

\draw[->] (A_empty_yellow) -- (A_minus_yellow);
\draw[->] (D_minus) -- (A_minus_yellow);

\end{tikzpicture}
    \caption{A Hasse diagram showing the 12 subsets $A_*^*$, $D_*$, $U_*$ and possible arrows between them}
    \label{fig:mu(Z_2)=1}
\end{figure}
    
    Observe that $X_+$ is not connected to $X_-$, because $D_+$ is not connected to $D_-$, $D_+$ is not connected to $A_\emptyset$ or $A_-$, $D_-$ is not connected to $A_\emptyset$ or $A_+$, $A^+$ is not connected to $A^-$.
    
    $Y_+$ is not connected to $Y_-$, because $U_+$ is not connected to $U_-$, $U_+$ is not connected to $A_\pm$ or $A_-$, $U_-$ is not connected to $A_\pm$ or $A_+$, $A^+$ is not connected to $A^-$.
    
    No element of $X_+\sqcup X_-$ is above an element of $Y_-\sqcup Y_+$. Let us show this assuming without loss of generality that $x\in X_+$:
    \begin{itemize}
    \item
    $x\in D_+$ cannot be above any element of $A$ and above any element of $U$;
    \item
    $x\in A_+^+$ cannot be above any element of $U$ or above any element of $A^-\cup A_\pm\cup A_-$;
    \item
    $x\in A_\emptyset^+$ cannot be above any element of $U$ or above any element of $A_+\sqcup A_-\sqcup A_\pm$ or above any element of $A_\emptyset^-$.
    \end{itemize}
    So we construct a $G$-equivariant monotone map
    \[
    \phi(X_+) = (+,0),\quad \phi(X_-) = (-,0),\quad \phi(Y_+) = (+,1),\quad \phi(Y_-) = (-,1)
    \]
    that certifies $\xind P\le 1$.
\end{proof}

\subsection{Cross-index of a union in general}

We are now prepared to present the proof in the general setting. Before proceeding, we record a useful observation that will be employed repeatedly in both parts. 

\begin{lemma}
\label{lemma:move-down}
For a $G$-poset $A$, we may $G$-invariantly decompose
\begin{equation}
\label{equation:xind-decomposition}
A \;=\; A_0 \sqcup A_1 \sqcup \dots \sqcup A_{\xind A}
\end{equation}
in such a way that 

1) $\xind A_i = 0$ for any $i$;

2) every element of $A_i$ is either strictly less than or incomparable with every element of $A_j$ for any $j>i$;

3) for any $i \ge 1,\ldots,\xind A$ and any element $y \in A_i$, there exists an element $x \in A_{i-1}$ such that $x\prec y$.
\end{lemma}

In the remaining proofs we assume that the decomposition of Lemma~
\ref{lemma:move-down} takes place.

\begin{proof}
By Proposition~\ref{proposition:partition-xind} we have \eqref{equation:xind-decomposition} satisfying properties (1) and (2) from the list.  

If (3) is violated for some $y \in A_i$ (and for any $gy$ in its orbit $Gy$ by $G$-invariance of the decomposition~\eqref{equation:xind-decomposition}), then let $A_i(y) = \{y'\in A_i : y'\prec y\}$. Let us move all $gA_i(y)=A_i(gy)$, $g\in G$, to $A_{i-1}$. In terms of the corresponding map $P\to Q_{\xind P} G$, we preserve the signs (the $g$ of $(g, i)\in Q_{\xind P}G$) and decrease the value (the $i$ of $(g, i)\in Q_{\xind P}G$) for $\bigcup_{g\in G} A_i(gy)$ under such a move by $1$. The resulting $G$-map is clearly $G$-equivariant and well-defined, since any pair of sets $A_i(gy)$ and $A_i(g'y)$ is pairwise disjoint for $g\neq g'$ ($gy$ is not connected to $g'y$ in $A_i$). The monotonicity, property (2), is not violated since by transitivity $x\prec y'$ for $x\in A_{i-1}$ and $y'\in A_i(gy)$ is impossible. 

Moreover, such moves can be made for all $y\in A_i$ that do not have $x\in A_{i-1}$ such that $x\prec y$ at once. Indeed, the $G$-map $A_i\to G$ certifying $\xind A_i=0$ just becomes a part of the well-defined monotone $G$-map $A_{i-1}\to G$ certifying $\xind A_{i-1}=0$. Doing such moves for $i=1,\ldots,\xind A$ one achieves property (3).
\end{proof}

\begin{proof}[Proof of Theorem~\ref{Main:Theorem 1} (a)]

We may reduce the general situation to the case \(\xind B = 0\) with \(\xind A\) arbitrary. Indeed, we can decompose \(B = B_0 \sqcup B_1 \sqcup \dots \sqcup B_{\xind B}\) with respect to some \(G\)-map on \(B\), and then infer the general case 
\begin{multline*}
\xind A\cup (B_0\sqcup\dots\sqcup B_{\xind B}) \le \xind A\cup (B_0\sqcup\dots\sqcup B_{\xind B-1}) + 2 \le \\
\dots \le \xind A\cup (B_0\sqcup\dots\sqcup B_{\xind B-k}) + 2k\le
\dots \le \xind A + 2(\xind B + 1)
\end{multline*}
only using the case $\xind B_k=0$ in the chain of inequalities. Before establishing the theorem in the case \(\xind B = 0\), we recall the argument for the special case \(\xind A = \xind B = 0\) from the proof of Theorem~\ref{thm:bound <=2}.

\begin{itemize}
    \item First, we transfer to \(A\) all elements \(b \in B\) that lie strictly between two elements \(a, a' \in A\), i.e., those satisfying \(a \prec b \prec a'\).
    \item Next, we decompose \(B\) as a disjoint union \(B = U \sqcup D\), where \(U\) consists precisely of those elements that are strictly greater than at least one element of \(A\), and \(D\) consists of those elements that are either strictly smaller than at least one element of \(A\) or incomparable with every element of \(A\).
\end{itemize}

We now show that the general case of arbitrary cross-index \(\xind A\) can be reduced to the special case \(\xind A = 0\), while \(\xind B\) remains equal to \(0\). Specifically, we argue by induction on \(\xind A\). The base case of this induction has already been established in Theorem~\ref{thm:bound <=2}. To make a step consider the decomposition \(A = A_0 \sqcup A_1 \sqcup \dots \sqcup A_{\xind A}\) induced by a given \(G\)-map on \(A\). Since \(\xind A_0 = \xind B = 0\), we apply the preceding procedure to transfer certain elements of \(B\) into \(A_0\), while preserving the structural properties of the decomposition \(A = A_0 \sqcup A_1 \sqcup \dots \sqcup A_{\xind A}\). We then partition the remaining subset of \(B\) as \(B = D \sqcup U\). Note that elements in \(D \cup A_0\) are either smaller than or incomparable to elements in \(U \cup (A_1 \sqcup \dots \sqcup A_{\xind A})\). The comparison between elements in \(D \cup A_0\) and \(U\) is immediate from the construction, and so is the comparison between \(A_0\) and \(A_i\) for \(i > 0\). It remains to verify the comparison between elements in \(D\) and elements in \(A_i\) for \(i > 0\). Indeed, \(D\) cannot contain an element \(u\) larger than an element \(v \in A_i\), since otherwise Lemma~\ref{lemma:move-down} implies that \(v\) must be larger than some element \(w \in A_0\). Hence \(u\) would be larger than \(w \in A_0\), contradicting the assumption \(u \in D\). For the \(G\)-map, we assign the two values \(0, 1\) to the elements in \(D\) and \(A_0\), respectively, and then apply the induction hypothesis to obtain the corresponding values, starting from \(2\), for elements in \(U \cup (A_1 \sqcup \dots \sqcup A_{\xind A})\), since \(\xind(A_1 \sqcup \dots \sqcup A_{\xind A})\) is smaller than \(\xind A\), while \(\xind U = 0\). This completes the induction, and the conclusion follows.
\end{proof}

\begin{proof}[Proof of Theorem~\ref{Main:Theorem 1} (b)]
As before, we may reduce the problem to the case \(\xind B = 0\) while allowing \(\xind A\) to be arbitrary. Specifically, we decompose \(B = B_0 \sqcup B_1 \sqcup \dots \sqcup B_{\xind B}\) with respect to a suitable \(G\)-map on \(B\), and then infer the general case 
\begin{multline*}
\xind A\cup (B_0\sqcup\dots\sqcup B_{\xind B}) \le \xind A\cup (B_0\sqcup\dots\sqcup B_{\xind B-1}) + 1 \le \\
\dots \le \xind A\cup (B_0\sqcup\dots\sqcup B_{\xind B-k}) + k\le
\dots \le \xind A + \xind B + 1
\end{multline*}
only using the case $\xind B_k=0$ in the chain of inequalities. Before proving the theorem in the case \(\xind B = 0\), we briefly recall the argument in the special case \(\xind A = \xind B = 0\), as presented in Theorem~\ref{thm: mu(Z_2)=1}.
\begin{enumerate}
    \item First, we move into \(A\) all elements \(b \in B\) for which there exist elements \(a, a' \in A\) such that \(a \prec b \prec a'\).
    \item Next, we decompose \(B\) as a disjoint union \(B = U \sqcup D\), where \(U\) consists of those elements that are strictly greater than at least one element of \(A\), and \(D\) consists of those elements that are either strictly smaller than at least one element of \(A\) or incomparable with every element of \(A\).
    \item Finally, we define \(X = D \cup A'\) and \(Y = A'' \cup U\), for some decomposition \(A = A' \sqcup A''\), in such a way that \(\xind X = \xind Y = 0\). Moreover, every element of \(X\) is either smaller than or incomparable to every element of \(Y\).
\end{enumerate}
As in the previous proof, we now reduce the case of arbitrary \(\xind A\) to the case \(\xind A = 0\), still assuming \(\xind B = 0\). We proceed by induction on \(\xind A\), assuming that the statement holds for all smaller values of the cross-index of \(A\). The base case has already been established in Theorem~\ref{thm: mu(Z_2)=1}. Choose a decomposition \(A = A_0 \sqcup A_1 \sqcup \dots \sqcup A_{\xind A}\) with respect to a \(G\)-map for \(A\). Since \(\xind A_0 = \xind B = 0\), we may apply the above procedure to \(A_0\) and \(B\). First, we move some elements of \(B\) into \(A_0\) in such a way that the decomposition \(A = A_0 \sqcup A_1 \sqcup \dots \sqcup A_{\xind A}\) continues to satisfy the required properties by transitivity. Then we define \(X = D \cup A_0'\) and \(Y = A_0'' \cup U\), for a decomposition \(A_0 = A_0' \sqcup A_0''\), so that \(\xind X = \xind Y = 0\). Here, \(B = D \sqcup U\) is constructed with respect to \(A_0\). Observe that every element of \(X\) is either smaller than or incomparable to every element of \(Y \cup (A_1 \sqcup \dots \sqcup A_{\xind A})\). The relation between elements of \(X\) and \(Y\) follows directly from the construction, as wells as the relation between the elements of $A_0'\subseteq X$ and $A_1 \sqcup \dots \sqcup A_{\xind A}$. It remains to verify the relation between elements of $D\subseteq X$ and the sets \(A_i\), $i\ge 1$. Indeed, \(D\) cannot contain an element \(u\) that is larger than some element \(v \in A_i\) with \(i \ge 1\), since in that case Lemma~\ref{lemma:move-down} would imply that \(v\) is larger than some element \(w \in A_0\). Hence \(u\) would be larger than \(w \in A_0\), contradicting the assumption that \(u \in D\). To construct the required \(G\)-map, we assign the value \(0\) to all elements of \(X\), and then apply the induction hypothesis to \(Y \cup (A_1 \sqcup \dots \sqcup A_{\xind A})\), assigning values starting from \(1\). This is possible because \(\xind(A_1 \sqcup \dots \sqcup A_{\xind A})\) is smaller than \(\xind A\), while \(\xind Y = 0\). This completes the inductive step, and the desired conclusion follows.
\end{proof}

Using the notation \(\mu(G)\), both inequalities in Theorem~\ref{Main:Theorem 1} can be unified into the following single inequality: for every group \(G\), if \(P = A \cup B\) where \(P\) is a \(G\)-poset and \(A,B\) are \(G\)-subposets of \(P\), then
\begin{equation}
\label{unified inequality}
\xind P \le \xind A + \mu(G)\bigl(\xind B + 1\bigr).
\end{equation}

Regarding sharpness, it is straightforward to verify that this bound is attained for \(G = \mathbb{Z}_2\). So we provide:

\begin{example}
Let
\(
P=Q_n\mathbb{Z}_2,
\)
so that \(\xind P=n\). For any integer \(m\) with \(0\le m\le n\), decompose \(P\) as
\[
A=\{(\pm,a):0\le a\le m\}=Q_m\mathbb{Z}_2
\quad\text{and}\quad
B=\{(\pm,b):m<b\le n\}\cong_{\mathbb{Z}_2} Q_{\,n-m-1}\mathbb{Z}_2.
\]
Then \(\xind A=m\), \(\xind B=n-m-1\), and therefore
\[
\xind P=n=\xind A+\xind B+1.
\]
This shows that the estimate in the case \(G=\mathbb{Z}_2\) is attained, and hence is best possible. 
\end{example}

Establishing sharpness for the other bound is more involved. We first construct an explicit example showing that \(\mu(G) = 2\) whenever \(G\) has at least three elements. This construction is then used to prove that the above inequality is best possible for \(G\)-posets with arbitrarily large cross-index.

\begin{theorem}
\label{Thm: mu(G)=2 & large gap}
If $G$ is a group with at least three elements, then $\mu(G)=2$. Furthermore, for any pair of non‑negative integers $m,n$ with $m \ge n$, there exists a $G$‑poset $Q$ that can be written as the union of two $G$‑subposets $Q_1$ and $Q_2$ such that  
\[
\xind Q_1 = m, \quad \xind Q_2 = n, \quad \text{and} \quad \xind Q = m + 2(n + 1).
\]
\end{theorem}

\begin{proof}
\label{Proof Main:Theorem 1}
To verify this claim, by Theorem~\ref{thm:bound <=2}, it suffices to construct a $G$-poset 
$P$ with $\xind P\ge 2$ such that it can be decomposed into the union of two 
$G$-subposets, each with cross-index $0$. Let $A, B, C_1, C_2, C_3, D$, and $E$ be seven disjoint copies of $G$, and define
\[
P = A \sqcup B \sqcup C_1 \sqcup C_2 \sqcup C_3 \sqcup D \sqcup E.
\]
For notational clarity, we write $g_X$ to denote the copy of an element $g \in G$ within the component $X$, where $X$ is any of the sets $A$, $B$, $C_1$, $C_2$, $C_3$, $D$, or $E$ defined above. Note that the group $G$ acts naturally on $P$ via its binary operation on each component, endowing $P$ with the structure of a $G$-set. To equip $P$ with a $G$-poset structure, we fix three distinct elements $g_1, g_2, g_3 \in G$ and define the partial order according to the following diagram: 

\begin{figure}[h!]
        \centering
        \includegraphics[width=0.5\textwidth]{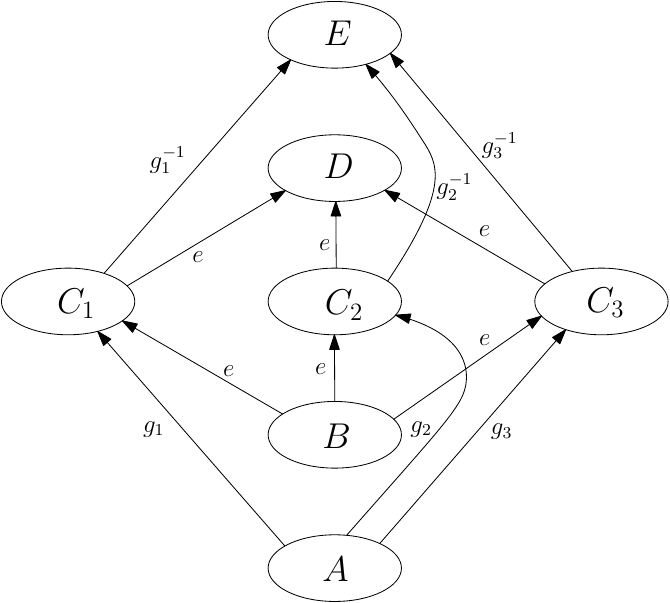}
        \caption{An example showing $\mu(G)=2$ for $|G|\ge 3$}
        \label{fig:union-counterex}
\end{figure}
More precisely, in the above diagram, an arrow labeled $g$ from $X$ to $Y$ indicates the relation $e_X \prec g_Y$. We then minimally extend these relations to ensure $P$ becomes a $G$-poset.

We now prove that $\xind P\ge 2$. Suppose, for contradiction, that there exists a $G$-map $\psi \colon P \to Q_1G$.  Let $P = P_1 \cup P_2$ where:
\begin{align*}
P_1 &= A \sqcup E, \\
P_2 &= B \sqcup C_1 \sqcup C_2 \sqcup C_3 \sqcup D.
\end{align*}
Note that $\xind P_1 = \xind P_2 = 0$ by construction. Without loss of generality, we may assume all elements of $A$ and $B$ take the value $0$ and all elements of $D$ and $E$ take the value $1$ under $\psi$, as $A$ and $B$ consist of the minimal elements of $P$, and $D$ and $E$ consist of the maximal elements. By the pigeonhole principle, at least two of $C_1$, $C_2$, and $C_3$ must share the same value. Due to the symmetry in the middle-layer posets, we may assume without loss of generality that $\psi$ assigns the same value $v \in \{0,1\}$ to all elements of both $C_2$ and $C_3$. We now consider two scenarios:
\begin{itemize}
    \item \textbf{Case 1:} $v= 0$. Consider the quadrilateral $AC_3BC_2$ where all edges are links and orbits share the same value.  The comparability graph of $P$ contains a path $e_{A}\to (g_3)_{C_3}\to (g_3)_{B}\to (g_3)_{C_2}\to (g_2^{-1}g_3)_{A}$.
    This forces $\psi(e_{A}) = \psi((g_2^{-1}g_3)_{A})$ while $g_2 \neq g_3$ and $g_2^{-1}g_3\neq e$, yielding a contradiction. 
    \item \textbf{Case 2:} $v= 1$. The quadrilateral $C_3EC_2D$ leads to an analogous contradiction through the same reasoning.
\end{itemize}
Thus, $\xind P\ge 2$ as claimed. Consequently $\mu(G)=2$. For the second assertion, let \(P=P_1\cup P_2\) be the above example, and take \(n+1\) disjoint copies \(\mathcal P^{(0)},\mathcal P^{(1)},\dots,\mathcal P^{(n)}\) of \(P\), writing \(\mathcal P^{(i)}=\mathcal P_1^{(i)}\cup \mathcal P_2^{(i)}\) accordingly. Define
\(
Q:=\mathcal P^{(0)}*\mathcal P^{(1)}*\cdots *\mathcal P^{(n)}*Q_{m-n-1}G,
\)
with the convention \(Q_{-1}G=\varnothing\) when \(m=n\), and let
\[
Q_1:=\mathcal P_1^{(0)}*\cdots *\mathcal P_1^{(n)}*Q_{m-n-1}G,
\qquad\&\qquad
Q_2:=\mathcal P_2^{(0)}*\cdots *\mathcal P_2^{(n)}.
\]
By additivity of the cross-index under joins, \(\xind Q_1=m\), \(\xind Q_2=n\), and
\[
\xind Q
= \xind Q_1+\mu(G)(\xind Q_2+1)
= m+2(n+1).
\]
This proves the sharpness of the general bound for $|G|\ge 3$.
\end{proof}

\begin{remark}\label{remark: final}
In the construction above, the poset \(P\) has height \(2\), and therefore \(\xind P\le 2\) by Proposition~\ref{prop:dimension_bound}; hence \(\xind P=2\). On the other hand,
\(
\ind \Delta P =1
\). Indeed, \(\Delta P\) is a \(G\)-subcomplex of \(\Delta P_1 * \Delta P_2\), so by monotonicity and subadditivity of the topological index,
\[
\ind \Delta P
\le \ind(\Delta P_1 * \Delta P_2)
\le \ind \Delta P_1 + \ind \Delta P_2 +1
=1,
\]
since \(\xind P_1=\xind P_2=0\) implies \(\ind \Delta P_1=\ind \Delta P_2=0\).
On the other hand, \(\ind \Delta P \ge 1\), since \(\xind P \neq 0\). Hence \(\ind \Delta P = 1\). Thus, for every group \(G\) with \(|G|>2\), we obtain another concrete example with \(\ind \Delta P=1\) and \(\xind P=2\).
\end{remark}

\begin{remark}\label{remark: gap between simplicial index & cross index}
By Corollary~\ref{cor: main result}(a), the equality \(\sind \Delta P = \xind P\) holds for every free \(\mathbb{Z}_2\)-poset \(P\). In contrast, Corollary~\ref{cor: main result}(b) provides only the weaker bound \(\sind \Delta P \le \xind P\). We conclude this part by showing that, in this setting, the gap \( \xind P - \sind \Delta P\) may indeed occur and can be arbitrarily large. Let \(P = P_1 \cup P_2\) be the poset constructed in the proof of Theorem~\ref{Thm: mu(G)=2 & large gap}. As observed in Remark~\ref{remark: final}, we have \(\xind P = 2\).
We claim that \(\sind \Delta P = 1\). Indeed, define a \(G\)-simplicial map \(\Delta P \to E_1G\) by sending each vertex \(g \in P_1\) to \((g,1)\) and each vertex \(g \in P_2\) to \((g,0)\). It is straightforward to verify that this is \(G\)-equivariant simplicial map from \(\Delta P\) to \(E_1G\). Therefore,
\(
\sind \Delta P \le 1.
\)
On the other hand, \(\sind \Delta P \ge 1\), since \(\xind P \neq 0\). Hence \(\sind \Delta P = 1\). Now let \( Q= P * \cdots * P\)
be the \(n\)-fold join of \(P\) with itself. By additivity of the cross-index under joins,
\(
\xind Q = 3n - 1.
\)
On the other hand, by subadditivity of the simplicial index,
\(
\sind \Delta Q
\le 2n - 1.
\)

Thus, this family yields examples of \(G\)-posets \(P\) with arbitrarily large cross-index \(\xind P\) for which \(\sind \Delta P\) can be as small as approximately \(\frac{2}{3}\) of \(\xind P\). Moreover, in the case \(n=1\), our construction shows that the lower bound from Corollary~\ref{cor: main result}(b), namely \(\xind P /2\), can already be attained. It would therefore be interesting to determine whether this ratio can be reduced further to \(1/2\) for \(G\)-posets with arbitrarily large cross-index, thereby matching the known lower bound more generally. We state this open question explicitly below.
\end{remark}

\section{Open Problems}

We conclude by compiling open problems and conjectures that have emerged from our investigation, which we believe merit further study. A question arises from Remark in Remark~\ref{remark: gap between simplicial index & cross index} above:

\begin{question}
What is the right asymptotics for a lower bound of the ratio \(\frac{\sind \Delta P}{\xind P}\) for $G$-posets $P$ when \(\xind P\to \infty\)?
\end{question}

We repeat here Question~\ref{que: que3} for reader's convenience:

\begin{question*}[Question~\ref{que: que3}]
Let \( n > 1 \) be a positive integer. Does there exist a connected finite free \( G \)-poset \( P \) such that
\[
\ind \Delta P = 1 \quad \text{but} \quad \xind P = n?
\]
\end{question*}

At present, this question has been solved only in the case \(n=2\); see Example~\ref{ex:21} and, for \(|G|\ge 3\), also Remark~\ref{remark: final}. The remaining cases are open.

Given a $G$-simplicial complex $\mathcal{K}$, it was observed in the introduction that the sequence
\begin{equation}
\label{equation:face-poset-barycentric}
 \sind\mathcal{K},\, \xind \mathcal{F}(\mathcal{K}),\, \sind \sd(\mathcal{K}),\, \xind \mathcal{F}(\sd(\mathcal{K})),\, \dotsc   
\end{equation}
is nonincreasing and eventually becomes constant, equal to $\ind \mathcal{K}$. This naturally leads to the following questions:

\begin{question}
\label{que:index_gap_stabilization}
For a given $G$-complex $\mathcal{K}$, how many steps are required for the sequence \eqref{equation:face-poset-barycentric} to become constant?
\end{question}
\begin{question}
\label{que:index_gap_values}
What values can the gaps between successive terms of the sequence \eqref{equation:face-poset-barycentric} take?
\end{question}
At present, only the behavior of the first two terms is understood in general; see Proposition~\ref{proposition:chromatic}. Corollary~\ref{cor: main result} further controls the pairs
\[
(\xind \mathcal F(\sd^r\mathcal K),\, \sind(\sd^{r+1}\mathcal K)),
\]
yielding a complete answer for \(G=\mathbb Z_2\) and a partial answer for groups of order at least three. 

Finally, inequality~\eqref{eq:index_inequality 2} suggests the following question concerning the gap between the simplicial index and the topological index of a poset.

\begin{question}[Gap between the simplicial index and the topological index]
\label{que:index_gap index & simplicial index}
Given integers \(m>n\ge 1\), does there exist a finite free \(G\)-poset \(P\) such that
\[
\sind \Delta P=m
\qquad\text{and}\qquad
\ind \Delta P=n?
\]
\end{question}

For \(G=\mathbb Z_2\), Corollary~\ref{cor: main result} shows that this question is equivalent to Question~\ref{que: que2}, since in that case \(\sind \Delta P=\xind P\).


\bibliographystyle{plain}
\bibliography{biblio}

\end{document}